\documentclass[12pt,reqno]{amsart}

\calclayout

\usepackage[dvips]{graphics}
\usepackage{times}
\usepackage{mathrsfs}
\usepackage[T1]{fontenc}
\usepackage{latexsym}
\usepackage{epsfig}
\usepackage{hyperref}
\usepackage{amsmath,amsfonts,amsthm,amssymb,amscd}
\usepackage{pstricks}
\usepackage{tikz-cd}
\usetikzlibrary{calc, positioning, matrix, arrows, decorations.pathmorphing, shapes, backgrounds}
\usepackage{setspace}

\usepackage{caption}
\usepackage{subcaption}
\newcommand{\kommentar}[1]{}

\newcommand\be{\begin{equation}}
\newcommand\ee{\end{equation}}
\newcommand\bea{\begin{eqnarray}}
\newcommand\eea{\end{eqnarray}}
\newcommand\bi{\begin{itemize}}
\newcommand\ei{\end{itemize}}
\newcommand\ben{\begin{enumerate}}
\newcommand\een{\end{enumerate}}
\newcommand\bc{\begin{center}}
\newcommand\ec{\end{center}}
\newcommand\ba{\begin{array}}
\newcommand\ea{\end{array}}

\newcommand{\N}{\mathbb{N}}

\newtheorem{thm}{Theorem}[section]

\newtheorem{lem}[thm]{Lemma}
\newtheorem{prop}[thm]{Proposition}
\newtheorem{exa}[thm]{Example}
\newtheorem{defi}[thm]{Definition}

\theoremstyle{definition}

\newcommand{\xt}{\widetilde{x}}

\newcommand{\Spec}{\text{Spec}\,}

\newcommand{\wh}[1]{\widehat{#1}}

\numberwithin{equation}{section}

\title{Noncatenary Unique Factorization Domains}
\author{Alexandra Bonat and S. Loepp}

\begin{document}
\maketitle


\pagestyle{plain}

\begin{abstract}

We demonstrate a class of local (Noetherian) unique factorization domains (UFDs) that are noncatenary at infinitely many places. In particular, if $A$ is in our class of UFDs, then the prime spectrum of $A$ contains infinitely many disjoint (except at the maximal ideal) noncatenary subsets. As a consequence of our result, there are infinitely many height one prime ideals $P$ of $A$ such that $A/P$ is not catenary.  We also construct a countable local UFD $A$ satisfying the property that for {\em every} height one prime ideal $P$ of $A$, $A/P$ is not catenary.
\end{abstract}


\section{Introduction}\label{section}

An overarching goal of commutative algebra is to understand the set of prime ideals (the prime spectrum) of a commutative ring. One way of understanding this set is to examine its structure as a partially ordered set (poset) under containment. Particularly, we are interested in whether or not a Noetherian ring with desirable algebraic properties can have prime spectra that behave strangely. One open question is, given a partially ordered set $X$, when is $X$ isomorphic to the prime spectrum of a Noetherian ring? This question has not been answered fully, although much progress has been made (see, for example, \cite{wiegands} for a nice survey). Relatedly, we can focus on a subset of the prime spectrum of a Noetherian ring and ask the question: given a partially ordered set $X$, can it be embedded into the prime spectrum of a Noetherian ring in a way that preserves saturated chains? It was previously conjectured that a Noetherian ring must be catenary. However, in 1956, this conjecture was disproved by Nagata in \cite{nagata_1956} when he constructed a noncatenary Noetherian domain. Then, in 1979, Heitmann proved in \cite{heitmann1979} that, given a finite partially ordered set $X$, there exists a Noetherian domain $R$ such that $X$ can be embedded into the prime spectrum of 
 $R$ in a way that preserves saturated chains. This shows that there is no limit to ``how noncatenary'' a Noetherian ring can be.

These results regarding noncatenary Noetherian rings raise the question of whether or not a Noetherian ring with nice algebraic properties can have unusual posets  embedded in its prime spectrum. In particular, one might ask whether or not a Noetherian unique factorization domain (UFD) can be noncatenary. The answer to this question was not known until Heitmann constructed a noncatenary UFD in \cite{heitmann} in 1993. Later, more examples of families of noncatenary UFDs were constructed (see, for example, \cite{small} and \cite{semendinger}). Finally, in \cite{colbert2022finite} it was shown, that, similar to Heitmann's result in \cite{heitmann1979} for Noetherian rings, given a finite partially ordered set $X$, there exists a Noetherian UFD $R$ such that $X$ can be embedded into the prime spectrum of $R$ in a way that preserves saturated chains.  It is natural to ask whether or not this result holds if $X$ is countably infinite. 


Particularly, in this article, we are interested in whether a Noetherian UFD can be ``infinitely noncatenary.'' More specifically, we ask if there exists a Noetherian UFD whose spectrum contains infinitely many disjoint noncatenary subsets.  We use ideas from \cite{semendinger}, \cite{heitmann}, and \cite{small} to show in Theorem \ref{constructionalt} that such a Noetherian UFD does indeed exist. A consequence of Theorem \ref{constructionalt} is that there exist Noetherian UFDs $A$ satisfying the property that for infinitely many height one prime ideals $P$ of $A$, the ring $A/P$ is not catenary. Moreover, in Theorem \ref{everyheightone}, we construct a four dimensional local (Noetherian) UFD such that the quotient ring at {\em every} height one prime ideal is noncatenary. Thus, we have shown that there exists a Noetherian UFD that is not only noncatenary at infinitely many places, but is, in some sense, noncatenary everywhere.

Our strategy is to start with a complete local (Noetherian) ring $T$ that is not equidimensional.  We then construct a subring $A$ of $T$ such that $A$ is a UFD whose completion is $T$ and such that there are infinitely many pairs of height one prime ideals $P$ and $Q$ of $T$ such that $P$ and $Q$ have different coheights and $A \cap P = A \cap Q \neq (0)$.  We also adjoin generators of carefully chosen prime ideals of $T$ to $A$ to ensure that there are infinitely many noncatenary subsets of prime ideals of $A$.

A motivating example is the complete local ring $T = k[[x,y,z,w,t]]/((x) \cap (y, z))$ where $k$ is a field and $x,y,z,w,t$ are indeterminates. A consequence of our main result, Theorem \ref{constructionalt}, is that $T$ has a subring that is a local (Noetherian) UFD whose prime spectrum contains the poset in Figure \ref{fig:exampleposet}, which is the poset in Figure \ref{fig:examplefiniteposet} repeated infinitely many times (see Example \ref{motivatingexample}). We also show in Theorem \ref{everyheightone} that in the case where $k = \mathbb{Q}$, $T$ has a subring that is a local (Noetherian) UFD such that every height one prime ideal is contained in a maximal saturated chain of prime ideals of length three and a maximal saturated chain of prime ideals of length four. In other words, every height one prime ideal is contained in a subset of the prime spectrum that is isomorphic to the poset in Figure \ref{fig:examplefiniteposet}. 



In Section \ref{Nsubrings}, we recall the notion of an N-subring of a complete local ring, first introduced in \cite{heitmann}. The bulk of Section \ref{Nsubrings} is devoted to proving results about N-subrings that are essential for our constructions. We state and prove our main result, Theorem \ref{constructionalt}, in Section \ref{MainResult}. In particular, we demonstrate a class of local (Notherian) UFDs such that, if $A$ is in the class, then the prime spectrum of $A$ contains infinitely many disjoint (except at the maximal ideal) noncatenary subsets. Finally, in Section \ref{noncatenaryUFD}, we construct a four dimensional countable local (Noetherian) UFD $A$ such that, for every height one prime ideal $P$ of $A$, the ring $A/P$ is not catenary. 

 \begin{figure}
         \centering
                 \begin{tikzpicture}  
  [scale=.9,auto=center,every node/.style={circle,fill=black}] 
    
   \node (a1) at (2,1) {};  
  \node (a2) at (2,2) {}; 
  \node (a3) at (1,3) {};  
  \node (a4) at (1,4) {};  
  \node (a5) at (2,5)  {};  
  \node (a6) at (3,3.5)  {};  

  \node (a7) at (0,2) {};  
  \node (a8) at (0,3.7) {};  
  \node (a9) at (-1,3)  {};  
  \node (a10) at (-1,4)  {}; 

  \node (a11) at (3.5,2) {};  
  \node (a12) at (4,3) {};  
  \node (a13) at (4,4)  {};  
  \node (a14) at (5.7,4)  {}; 

    \node (a15) at (-2,3) {};  
  \node (a16) at (-2.5,3) {};  
  \node (a17) at (-3,3)  {};  

   \node (a18) at (6.5,3) {};  
  \node (a19) at (7,3) {};  
  \node (a20) at (7.5,3)  {};

  \draw[draw, line width=2pt] (a1) -- (a2); 
  \draw[draw, line width=2pt] (a2) -- (a3);  
  \draw[draw, line width=2pt] (a3) -- (a4);  
  \draw[draw, line width=2pt] (a4) -- (a5);  
  \draw[draw, line width=2pt] (a5) -- (a6);  
  \draw[draw, line width=2pt] (a6) -- (a2);  

  \draw[draw, line width=2pt] (a1) -- (a7); 
  \draw[draw, line width=2pt] (a7) -- (a8);  
  \draw[draw, line width=2pt] (a8) -- (a5);  
  \draw[draw, line width=2pt] (a5) -- (a10);  
  \draw[draw, line width=2pt] (a10) -- (a9);  
  \draw[draw, line width=2pt] (a9) -- (a7); 

  \draw[draw, line width=2pt] (a1) -- (a11); 
  \draw[draw, line width=2pt] (a11) -- (a12);  
  \draw[draw, line width=2pt] (a12) -- (a13);  
  \draw[draw, line width=2pt] (a13) -- (a5);  
  \draw[draw, line width=2pt] (a5) -- (a14);  
  \draw[draw, line width=2pt] (a14) -- (a11);

\end{tikzpicture} 
         \caption{}
         \label{fig:exampleposet}
     \end{figure}

     \begin{figure}
         \centering
                 \begin{tikzpicture}  
  [scale=.9,auto=center,every node/.style={circle,fill=black}] 
    
   \node (a1) at (2,1) {};  
  \node (a2) at (2,2) {}; 
  \node (a3) at (1,3) {};  
  \node (a4) at (1,4) {};  
  \node (a5) at (2,5)  {};  
  \node (a6) at (3,3.5)  {};

  \draw[draw, line width=2pt] (a1) -- (a2); 
  \draw[draw, line width=2pt] (a2) -- (a3);  
  \draw[draw, line width=2pt] (a3) -- (a4);  
  \draw[draw, line width=2pt] (a4) -- (a5);  
  \draw[draw, line width=2pt] (a5) -- (a6);  
  \draw[draw, line width=2pt] (a6) -- (a2);

\end{tikzpicture} 
         \caption{}
         \label{fig:examplefiniteposet}
     \end{figure}

\section{Preliminaries}\label{Nsubrings}

We first establish some terminology. In this paper, all rings are assumed to be commutative with unity.  
When $R$ is a Noetherian ring with exactly one maximal ideal, we say $R$ is \textit{local}. If $R$ has exactly one maximal ideal but is not necessarily Noetherian, we call it \textit{quasi-local}. We use the notation $(R, M)$ when $R$ is a quasi-local ring and $M$ is the maximal ideal of $R$.  If $(R,M)$ is a local ring, we use $\widehat{R}$ to denote the $M$-adic completion of $R$.
We use the standard notation Spec$R$ to denote the set of prime ideals of the ring $R$. If $R$ is a local ring, we informally refer to two prime ideals $P, Q \in \Spec \wh R$ as being ``glued together" in $R$ if $P \cap R = Q \cap R$. 

Before we begin, we recall what it means for a ring to be catenary.

\begin{defi}
    A ring $R$ is called \textit{catenary} if for all $P, Q \in \Spec R$ with $P \subseteq Q$, every saturated chain of prime ideals from $P$ to $Q$ has the same length. 
\end{defi}

If a ring is not catenary, we say that it is noncatenary. The following prime avoidance lemma will be used multiple times in our construction.

\begin{lem}[\cite{heitmann}, Lemma 2] \label{countableprimeavoidance}
    Let $T$ be a complete local ring with maximal ideal $M$, let $C$ be a countable set of prime ideals in $\Spec T$ such that $M \notin C$, and let $D$ be a countable set of elements of $T$. If $I$ is an ideal of $T$ which is contained in no single $P$ in $C$, then $I \not \subset \bigcup \{(P + r) : P \in C, r \in D\}$.
\end{lem}

To construct our noncatenary UFDs, we use ideas and results from \cite{heitmann} which rely heavily on a specific type of subring of a complete local ring, called an N-subring.

\begin{defi}[\cite{heitmann}] 
    Let $(T, M)$ be a complete local ring. We say a quasi-local subring $(R, M \cap R)$ of $T$ is an \textit{N-subring} of $T$ if $R$ is a UFD and 

    \begin{enumerate}
        \item $|R| \leq \text{max}\{\aleph_0, |T/M| \}$, with equality only when $T/M$ is countable,
        \item $Q \cap R = (0)$ for all $Q \in \text{Ass} T$, and
        \item If $t \in T$ is regular and $P \in \text{Ass} (T/tT)$, then $\text{ht} (P \cap R) \leq 1$.
    \end{enumerate}
\end{defi}

The remainder of this section will be devoted to proving results about N-subrings that will be crucial for our constructions.

The following lemma establishes machinery to adjoin an element of $T$ to an N-subring of $T$ and ensure that the result is still an N-subring of $T$. Note that, if $R$ is a subring of a ring $T$ and $P$ is a prime ideal of $T$, then there is an injective map from $R/(R \cap P)$ to $T/P$. Thus, it make sense to say that an element $x + P \in T/P$ is either algebraic or transcendental over the ring $R/(R \cap P).$ 

\begin{lem} [\cite{loepp}, Lemma 11] \label{loepp}
 Let $(T, M)$ be a complete local ring and let $\mathfrak{p} \in \textnormal{Spec} T$. Let $R$ be an N-subring of $T$ with $\mathfrak{p} \cap R = (0)$. Suppose $C \subset \textnormal{Spec} T$ satisfies the following conditions: 

\begin{enumerate}
    \item $M \notin C$,
    \item $\mathfrak{p} \in C$,
    \item $\{P \in \textnormal{Spec} T \ | \ P \in \textnormal{Ass} (T/rT) \textnormal{ with } 0 \neq r \in R \} \subset C$, and
    \item $\textnormal{Ass} (T) \subset C$.
\end{enumerate}
Let $x \in T$ be such that $x \notin P$ and $x + P$ is transcendental over $R/(R \cap P)$ as an element of $T/P$ for every $P \in C$. Then, $S = R[x]_{(R[x] \cap M)}$ is an N-subring of $T$ properly containing $R$, $|S| = \textnormal{max}\{\aleph_0, |R|\}$, and $S \cap \mathfrak{p} = (0)$.
\end{lem}

In our construction, we ensure that our final UFD contains generating sets of carefully chosen prime ideals of $T$. 
 We use the following lemma to make sure that our UFD contains a generating set for one prime ideal of $T$ satisfying very specific properties. 

\begin{lem}\label{onegeneratingset}
     Let $(T, M)$ be a complete local ring with depth$T \geq 2$, let $R$ be a countable N-subring of $T$, and let $\mathfrak{p}$ be a nonmaximal prime ideal of $T$ such that $R \cap \mathfrak{p} = (0).$ Let $Q$ be a prime ideal of $T$ such that $Q \not\subseteq \mathfrak{p}$ and, if $P \in \Spec T$ such that $P \in \mbox{Ass}(T)$ or $P \in \mbox{Ass}(T/zT)$ for some nonzero regular element $z \in T$, then $Q \not\subseteq P$.
     Then there exists a countable N-subring $S$ of $T$ such that $R \subseteq S$, prime elements in $R$ are prime in $S$, $S$ contains a generating set for $Q$, and $S \cap \mathfrak{p} = (0)$. 
\end{lem}

\begin{proof}
Let $Q = (a_1,a_2, \ldots ,a_m)$ for $a_i \in T$. Define $$C = \mbox{Ass}T \cup \{ P \in \Spec T \, | \, P \in \mbox{Ass}(T/rT) \mbox{ for some } 0 \neq r \in R \} \cup \{\mathfrak{p}\}.$$ Note that, since depth$T\geq 2$, $M \not\in C$ and, since $R$ is countable, so is $C$. In addition, our hypotheses imply that $Q \not\subseteq P$ for all $P \in C$. Use Lemma \ref{countableprimeavoidance} with $D = \{0\}$ to find $q_1 \in Q$ with $q_1 \not\in P$ for every $P \in C$.

If $P \in C$ and $a_1 + tq_1 + P = a_1 + t'q_1 + P$ for $t,t' \in T$ then $q_1(t - t') \in P$.  Since $q_1 \not\in P$, we have $t - t' \in P$ and so $t + P = t' + P$.  It follows that if $t + P \neq t' + P$ then $a_1 + tq_1 + P \neq a_1 + t'q_1 + P$. Now let $D_{(P)}$ be a full set of coset representatives for the cosets $t + P \in T/P$ that make $(a_1 + q_1t) + P$ algebraic over $R/(R \cap P)$. Since the algebraic closure of $R/(R \cap P)$ in $T/P$ is countable, we have that $D_{(P)}$ is countable. Let $D = \bigcup_{P \in C}D_{(P)}$, and note that $D$ is countable. Use Lemma \ref{countableprimeavoidance} to find $m_1 \in M$ such that $m_1 \not\in \bigcup \{(P + r) \, | \, P \in C, r \in D\}$. It follows that $(a_1 + q_1m_1) + P$ is transcendental over $R/(R \cap P)$ for all $P \in C$. By Lemma \ref{loepp}, if $\tilde{a}_1 = a_1 + q_1m_1$ then $R_1 = R[\tilde{a}_1]_{(R[\tilde{a}_1] \cap M)}$ is a countable N-subring of $T$ with $R_1 \cap \mathfrak{p} = (0)$. Let $P \in \mbox{Ass}T$. Then $\tilde{a}_1 + P$ is transcendental over $R/(R \cap P)$ and $R \cap P = (0)$.  It follows that $\tilde{a}_1$ is transcendental over $R$ and so prime elements in $R$ are prime in $R_1$. Note that $(\tilde{a}_1, a_2, \ldots ,a_m) + MQ = Q$ and so by Nakayama's Lemma, $Q = (a_1, a_2, \ldots ,a_m) = (\tilde{a}_1, a_2, \ldots ,a_m).$ 

Repeat the above process with $R$ replaced by $R_1$ to find $q_2 \in Q$ and $m_2 \in M$ so that, if $\tilde{a}_2 = a_2 + q_2m_2$ then $R_2 = R_1[\tilde{a}_2]_{(R_1[\tilde{a}_2] \cap M)}$ is a countable N-subring of $T$, $R_2 \cap \mathfrak{p} = (0)$, prime elements in $R_1$ are prime in $R_2$, and $Q = (a_1, a_2, \ldots ,a_m) =  (\tilde{a}_1, a_2, \ldots ,a_m) = (\tilde{a}_1, \tilde{a}_2, \ldots ,a_m).$

Continue the process to find a countable N-subring $R_m$ of $T$ such that $R \subseteq R_m$, $R_m \cap \mathfrak{p} = (0)$, prime elements in $R$ are prime in $R_m$ and $R_m$ contains a generating set for $Q$. Then $S = R_m$ is the desired N-subring of $T$.
\end{proof}

The following definition, taken from \cite{heitmann}, describes a particular type of extension of an N-subring of $T$ that will be useful for our construction.

\begin{defi} [\cite{heitmann}]
    Let $(T, M)$ be a complete local ring and let $R \subseteq S$ be N-subrings of $T$. We say that $S$ is an \textit{A-extension} of $R$ if prime elements in $R$ are prime in  $S$ and $|S| \leq \text{max}\{|R|, \aleph_0\}$.
\end{defi}

The next lemma show that, in certain circumstances, the union of an increasing chain of N-subrings of $T$ is again an N-subring of $T$. We use this result multiple times in our construction.

\begin{lem} [\cite{heitmann}, Lemma 6] \label{union}
    Let $(T, M)$ be a complete local ring and let $R_0$ be an N-subring of $T$. Let $\Omega$ be a well-ordered set with least element $0$ and assume either $\Omega$ is countable or for all $\alpha \in \Omega$, $|\{\beta \in \Omega \ | \ \beta < \alpha \} | < |T/M|$. Let $\gamma(\alpha) = \text{sup}\{\beta \in \Omega \ | \ \beta < \alpha \}$. Suppose $\{R_{\alpha} \ | \ \alpha \in \Omega \}$ is an ascending collection of rings such that if $\gamma(\alpha) = \alpha$, then $R_{\alpha} = \bigcup_{\beta < \alpha} R_{\beta}$ while if $\gamma(\alpha) < \alpha$, $R_{\alpha}$ is an A-extension of $R_{\gamma (\alpha)}$. 
    Then $S = \bigcup R_{\alpha}$ satisfies all conditions to be an N-subring of $T$ except the cardinality condition. $|S| \leq \text{max}\{\aleph_0, |R_0|, |\Omega|\}$. Further, elements which are prime in some $R_{\alpha}$ remain prime in $S$. 
\end{lem}

The following lemma shows that we can repeat Lemma \ref{onegeneratingset} infinitely many times. In other words, given a countable set of prime ideals of $T$ satisfying certain conditions, the lemma allows us to adjoin a generating set for each of these prime ideals to an N-subring of $T$, with the resulting ring being an N-subring of $T$.

\begin{lem} \label{generatingsets}
     Let $(T, M)$ be a complete local ring with depth$T \geq 2$, let $R$ be a countable N-subring of $T$, and let $\mathfrak{p}$ be a nonmaximal prime ideal of $T$ with $R \cap \mathfrak{p} = (0)$. Let $\{Q_j\}_{j \in \mathbb{N}}$ be a countable set of prime ideals of $T$ such that for every $j \in \mathbb{N}$, $Q_j \not\subseteq \mathfrak{p}$ and if $P \in \Spec T$ such that $P \in \mbox{Ass}(T)$ or $P \in \mbox{Ass}(T/zT)$ for some nonzero regular element $z \in T$, then $Q_j \not\subseteq P$.
     Then there exists a countable N-subring $S$ of $T$ such that $R \subseteq S$, $S \cap \mathfrak{p} = (0)$, prime elements in $R$ are prime in $S$, and, for every $j \in \mathbb{N}$, $S$ contains a generating set for $Q_j$. 
\end{lem}

\begin{proof}
    Let $R_0 = R$. By Lemma \ref{onegeneratingset}, there exists a countable N-subring $R_1$ of $T$ such that $R_0 \subseteq R_1$, $R_1 \cap \mathfrak{p} = (0)$, $R_1$ contains a generating set for $Q_1$, and $R_1$ is an A-extension of $R_0$. 
    We inductively define $R_k$ for every $k > 1$. Assume that $k > 1$ and  $R_{k- 1}$ has been defined so that for $\ell \leq k - 1$, $R_{\ell}$ is a countable N-subring of $T$ containing generating sets for $Q_1, Q_2, \ldots , Q_{\ell}$, $R_{\ell} \cap \mathfrak{p} = (0)$, and $R_{\ell}$ is an A-extension of $R_{\ell - 1}$. 
    Define $R_k$ to be the countable N-subring obtained from Lemma \ref{onegeneratingset} so that $R_k$ is an A-extension of $R_{k - 1}$, $R_k$ contains a generating set for $Q_k$, and $R_k \cap \mathfrak{p} = (0)$.  
    
    Let $S = \bigcup_{k \in \N}R_k$. We claim that this is the desired N-subring of $T$. For all $k \in \N$, $R_k$ is an A-extension of $R_{k-1}$, so by Lemma \ref{union} $S$ is a countable N-subring of $T$ such that prime elements in $R$ are prime in $S$. Furthermore, a generating set for $Q_k$ is contained in $R_k$, so, for every $j \in \mathbb{N}$, $S$ contains a generating set for $Q_j$. Finally, since $R_j \cap \mathfrak{p} = (0)$ for every $j \in \mathbb{N}$, we have $S \cap \mathfrak{p} = (0)$.
\end{proof}


We use the next result to identify height one prime ideals of $T$ that will be glued together in our final UFD.

\begin{lem}\label{findingheightones}
Let $(T,M)$ be a complete local ring and let $R$ be a countable N-subring of $T$.  Let $Q$ be a prime ideal of $T$ such that if $P \in \mbox{Ass}(T)$ or $P \in \mbox{Ass}(T/rT)$ for some $0 \neq r \in R$, then $Q \not\subseteq P$. 
Let $X = \{Q_1, Q_2,\ldots ,Q_n\}$ be a (possibly empty) set of prime ideals of $T$ such that $Q \not\subseteq Q_j$ for all $j = 1,2, \ldots ,n$. Then there exists a height one prime ideal $P'$ of $T$ such that $P' \subseteq Q$ and, if $P \in \mbox{Ass}(T)$ or $P \in \mbox{Ass}(T/rT)$ for some $0 \neq r \in R$, then $P' \not\subseteq P$. Moreover, if $X$ is not empty then $P' \not\subseteq Q_j$ for all $j = 1,2, \ldots ,n$.
\end{lem}

\begin{proof}
Use Lemma \ref{countableprimeavoidance} with $$C = \mbox{Ass}T \cup \{ P \in \Spec T \, | \, P \in \mbox{Ass}(T/rT) \mbox{ for some } 0 \neq r \in R \} \cup X,$$ $D = \{0\}$ and $I = Q$ to find $q \in Q$ such that $q \not\in Q_j$ for all $j = 1,2, \ldots ,n$ and, if $P \in \mbox{Ass}(T)$ or $P \in \mbox{Ass}(T/rT)$ for some $0 \neq r \in R$, then $q \not\in P$. Then $q$ is a nonzero regular element of $T$. Let $P'$ be a minimal prime ideal of $qT$ contained in $Q$. By the principal ideal theorem, ht$P' = 1$. Now suppose $P$ is a prime ideal of $T$ such that $P \in \mbox{Ass}(T)$ or $P \in \mbox{Ass}(T/rT)$ for some $0 \neq r \in R$, and assume $P' \subseteq P$. Then $q \in P' \subseteq P$, a contradiction. 
Similarly, if $P' \subseteq Q_j$ for some $j \in \{1,2, \ldots ,n\}$, then $q \in Q_j$, a contradiction.
\end{proof}

Note that, in the above lemma, if $Q$ contains only one minimal prime ideal of $T$ then $P'$ also contains only one minimal prime ideal of $T$.

In the next lemma, we show that, given an N-subring $R$ of $T$ and given certain height one prime ideals $P_1, \ldots ,P_s$ of $T$, we can adjoin a special element $\xt$ of $T$ to $R$ and obtain another N-subring of $T$.  Specifically, $\xt$ will satisfy the property that it is in $P_i$ for all $i = 1,2, \ldots ,s$. Our final UFD $A$ will satisfy the property that $A \cap P_i = \xt A$ for all $i = 1,2, \ldots ,s.$  In other words, the prime ideals $P_1, \ldots ,P_s$ will be glued together in $A$.


\begin{lem} \label{transcendental}
    Let $(T, M)$ be a complete local ring with depth$T \geq 2$ and suppose 
    $R$ is a countable N-subring of $T$. Let $P_1, \dots, P_s$ be height one prime ideals of $T$ such that, for every $i = 1,2, \ldots ,s$ we have that if $P \in \mbox{Ass}(T)$ or $P \in \mbox{Ass}(T/rT)$ for some $0 \neq r \in R$ then $P_i \not\subseteq P.$ 
   Let $X$ be a (possibly empty) finite set of prime ideals of $T$ such that $P_i \not\subseteq Q$ for every $Q \in X$ and for every $i = 1,2, \ldots ,s$. 
    Then, there exists $\xt \in \bigcap^s_{i = 1} P_i$ with $\xt \not\in \bigcup_{Q \in X}Q$ such that $S = R[\xt]_{(M \cap R[\xt])}$ is an N-subring of $T$ with $S \cap P_i = \xt S$ for every $i = 1,2, \ldots ,s.$ Moreover, prime elements in $R$ are prime in $S$.
\end{lem}

\begin{proof}
    Define $$C = \mbox{Ass}T \cup \{ P \in \Spec T \, | \, P \in \mbox{Ass}(T/rT) \mbox{ for some } 0 \neq r \in R \} \cup X$$
    and note that $C$ is countable and $M \not\in C$. Now, if $i  \in \{1,2, \ldots ,s\}$, we have $P_i \not\subseteq P$ for all $P \in C$.
    For every $i = 1,2, \ldots,s$, apply Lemma \ref{countableprimeavoidance}  to find $x_i \in P_i$ such that $x_i \notin P$ for all $P \in C$. Define $x = \prod_{i = 1}^s x_i$ and note that $x \in \bigcap_{i = 1}^s P_i$ and $x \not\in P$ for every $P \in C$. Now, fix $P \in C$ and let $t,t' \in T$. If $x(1 + t) + P = x(1 + t') + P$ as elements of $T/P$ then $x(t-t') \in P$. Since $x \not\in P$, we have $t-t' \in P$ and so $t + P = t' + P$.  It follows that if $t + P \neq t' + P$ then $x(1 + t) + P \neq x(1 + t') + P$. Let $D_{(P)}$ be a full set of coset representatives for the cosets $t + P$ that make $x(1 + t) + P$ algebraic over $R/(P \cap R)$, and note that $D_{(P)}$ is countable. Let $D = \bigcup_{P \in C} D_{(P)}$.
%
    Use Lemma \ref{countableprimeavoidance}, to find $\alpha \in M$ so that $x(1 + \alpha) + P \in T/P$ is transcendental over $R/(R \cap P)$ for every $P \in C$. Define $\xt = x(1 + \alpha)$. Then $\xt \not\in \bigcup_{Q \in X}Q$. By Lemma \ref{loepp}, $S = R[\xt]_{(M \cap R[\xt])}$ is an N-subring of $T$. Since $\xt$ is transcendental over $R$, prime elements of $R$ are prime in $S$.
    Fix $i \in \{1,2, \ldots ,s\}$. Since $\xt \in S$ and $\xt \in P_i$, we have $\xt S \subseteq S \cap P_i$. Since $R$ is a domain, $\xt S$ is a prime ideal of $S$. Now $\xt$ is a regular element of $T$ and $P_i$ is a height one prime ideal of $T$. It follows that $P_i \in \mbox{Ass}(T/\xt T)$. As $S$ is an N-subring of $T$ we have ht$(S \cap P_i) = 1$ and so $S \cap P_i = \xt S$.
\end{proof}

\section{The Main Result} \label{MainResult}

In this section, we construct the desired local UFD. We do so by first identifying infinitely many prime ideals $Q$ of our complete local ring $T$ that satisfy some specific properties, with the most important property being that dim$(T/Q) = 1$. We then note that a localization of the prime subring of $T$ is an N-subring, and we use Lemma \ref{generatingsets} to adjoin generators of our infinite set of identified prime ideals to obtain another N-subring of $T$. We then alternately use Lemma \ref{findingheightones} and Lemma \ref{transcendental} infinitely many times to find appropriate height one prime ideals of $T$ that will be glued together, along with elements that will generate the prime ideals of our UFD that the height one prime ideals will glue to. Finally, we rely on results from \cite{heitmann} to finish the construction our UFD $A$ and we show that $A$ contains our desired chains of prime ideals. We begin this process with a preliminary lemma.


\begin{lem}[\cite{small}, Lemma 2.8]\label{chains}
Let $(T,M)$ be a local ring with $M \not\in \mbox{Ass}(T)$, and let $P$ be a minimal prime ideal of $T$ with dim$(T/P) = n$. Then there exists a saturated chain of prime ideals of $T$, $P \subsetneq Q_1 \subsetneq \cdots \subsetneq Q_{n - 1} \subsetneq M$, such that, for each $i = 1,2, \ldots ,n - 1$, $Q_i \not\in \mbox{Ass}T$ and $P$ is the only minimal prime ideal contained in $Q_i$.
\end{lem}

We use Lemma \ref{findQ} to identify our infinite set of prime ideals of $T$ whose coheight is one.

\begin{lem} \label{findQ}
      Let $(T,M)$ be a local catenary ring with depth$(T)\geq 2$ and let $P_0$ be a minimal prime ideal of $T$ with dim$(T/P_0) = n \geq 3$. 
      Then there are infinitely prime ideals $Q$ of $T$ satisfying $P_0 \subseteq Q,$ $P_0$ is the only minimal prime ideal contained in $Q$, dim$(T/Q) = 1$, and if $P \in \Spec T$ such that $P \in \mbox{Ass}(T)$ or $P \in \mbox{Ass}(T/zT)$ for some nonzero regular element $z \in T$, then $Q \not\subseteq P$.
\end{lem}

\begin{proof}
Let $P_0 \subsetneq Q_1 \subsetneq \cdots \subsetneq Q_{n - 1} \subsetneq M$ be a saturated chain of prime ideals of $T$ obtained from Lemma \ref{chains} such that, for each $i = 1,2, \ldots ,n - 1$, $Q_i \not\in \mbox{Ass}T$ and $P_0$ is the only minimal prime ideal contained in $Q_i$. Suppose that $Q_{n - 2} \subseteq P$ for some $P \in \mbox{Ass}(T)$. As depth$(T) \geq 2$, $M \not\in \mbox{Ass}(T)$. Since $T$ is catenary and $Q_{n - 2} \not\in \mbox{Ass}(T)$, we have dim$(T/P)=1$. By Theorem 17.2 in \cite{matsumura}, $2 \leq \mbox{depth}(T) \leq \mbox{dim}(T/P) = 1$, a contradiction.  It follows that $Q_{n - 2} \not\subseteq P$ for all $P \in \mbox{Ass}(T)$. Thus, there exists a regular element $x \in T$ with $x \in Q_{n - 2}$.

Let $$X = \{Q \in \Spec(T) \, | \, Q_{n - 2} \subsetneq Q \subsetneq M \mbox{ is saturated} \}.$$ Since $T$ is Noetherian, $X$ has infinitely many elements. Suppose $Q \in X$ such that $Q$ contains $P_1$ where $P_1$ is a minimal prime ideal of $T$ satisfying $P_1 \neq P_0$. Then $Q$ is a minimal prime ideal of $Q_{n - 2} + P_1$, of which there are only finitely many.  The two sets $\mbox{Ass}(T)$ and $\mbox{Ass}(T/xT)$ are finite. Thus, the set $$Y = \{Q \in X \, | \, Q \not\in \mbox{Ass}(T), Q \not\in \mbox{Ass}(T/xT) \mbox{ and }$$ $$P_0 \mbox{ is the only minimal prime ideal of } T \mbox{ contained in } Q\}$$ contains infinitely many elements. Let $Q \in Y$. Then dim$(T/Q) = 1$ and if $P \in \mbox{Ass}(T)$ then $Q \not\subseteq P$.

Since $T_Q$ is a flat extension of $T$ and $x$ is a regular element of $T$, we have that $x$ is a regular element of $T_Q$. As $Q \not\in \mbox{Ass}(T/xT)$, the corollary to Theorem 6.2 in \cite{matsumura} gives that $QT_Q \not\in \mbox{Ass}(T_Q/xT_Q).$ It follows that depth$(T_Q) \geq 2$. Now suppose that $Q \subseteq P$ for some $P \in \mbox{Ass}(T/zT)$ where $z$ is a nonzero regular element of $T$. Then $P = M$ or $P = Q$. As depth$(T) \geq 2$, $M \not\in \mbox{Ass}(T/zT)$ and so $P = Q$. Therefore, $Q \in \mbox{Ass}(T/zT)$. It follows that $QT_Q \in \mbox{Ass}(T_Q/zT_Q)$, contradicting that depth$(T_Q) \geq 2$. Hence all elements of $Y$ satisfy the desired conditions.
\end{proof}

We are now ready to prove the main result of this paper.

\begin{thm} \label{constructionalt}
    Let $(T, M)$ be a complete local ring such that no integer of $T$ is a zerodivisor of $T$ and such that depth$(T) \geq 2$. Let $\{P_{0,1}, \dots, P_{0,s}\}$ be the minimal prime ideals of $T$ and suppose that for $i = 1,2, \ldots ,s,$ we have $\dim(T/P_{0,i}) = n_i \geq 3$. 
Then there exists a local UFD $(A, A \cap M)$ such that $\wh A = T$ and such that, for all $n \in \N$ and for all $i = 1,2, \ldots, s$, there exist saturated chains of prime ideals $(0) \subsetneq J_{1, n} \subsetneq J^{(i)}_{2, n} \subsetneq \dots \subsetneq J^{(i)}_{n_i - 1, n} \subsetneq M \cap A$ of $A$ satisfying $J^{(i)}_{a,b} = J^{(j)}_{c, d}$ if and only if $i = j$, $a = c$, and $b = d$, and $J_{1, n} = J_{1, m}$ if and only if $n = m$.
\end{thm}

\begin{proof}
Use Lemma \ref{findQ} to find prime ideals $Q_{n}^{(i)}$ of $T$ such that, for every $i = 1,2, \ldots, s$ and for all $n \in \mathbb{N}$, we have $Q_{n}^{(i)} = Q_{m}^{(i)}$ if and only if $n = m$, $P_{0,i} \subseteq Q_{n}^{(i)},$ $P_{0,i}$ is the only minimal prime ideal of $T$ contained in $Q_{n}^{(i)}$, dim$(T/Q_{n}^{(i)}) = 1$, and if $P \in \Spec T$ such that $P \in \mbox{Ass}(T)$ or $P \in \mbox{Ass}(T/zT)$ for some nonzero regular element $z \in T$, then $Q_{n}^{(i)} \not\subseteq P$.

Let $\Pi$ be the prime subring of $T$ and let $R$ be $\Pi$ localized at $\Pi \cap M$. Then $R$ is a countable N-subring of $T$. By Lemma \ref{generatingsets}, there is a countable N-subring $R_0$ of $T$ that contains a generating set for $Q_{n}^{(i)}$ for every $i = 1,2, \ldots ,s$ and for every $n \in \mathbb{N}$.

Now use Lemma \ref{findingheightones} letting $X$ be the empty set to find height one prime ideals $P_1^{(i)}$ of $T$ for each $i = 1,2, \ldots ,s$ satisfying $P_1^{(i)} \subseteq Q_{1}^{(i)}$ and, if $P \in \mbox{Ass}(T)$ or $P \in \mbox{Ass}(T/rT)$ for some $0 \neq r \in R_0$, then $P_1^{(i)} \not\subseteq P$. By Lemma \ref{transcendental} there exists $\xt_1 \in \bigcap^s_{i = 1} P_1^{(i)}$ such that $R_1 = R_0[\xt_1]_{(M \cap R_0[\xt_1])}$ is an N-subring of $T$ with $R_1 \cap P_1^{(i)} = \xt_1 R_1$ for every $i = 1,2, \ldots ,s.$ Moreover, prime elements of $R_0$ are prime in $R_1$. Note that $R_1$ is countable.

Use Lemma \ref{findingheightones} letting $X = \{Q_1^{(i)}\}$ to find height one prime ideals $P_2^{(i)}$ of $T$ for each $i = 1,2, \ldots ,s$ satisfying $P_2^{(i)} \subseteq Q_{2}^{(i)}$ and, if $P \in \mbox{Ass}(T)$ or $P \in \mbox{Ass}(T/rT)$ for some $0 \neq r \in R_1$, then $P_2^{(i)} \not\subseteq P$. Moreover, $P_2^{(i)} \not\subseteq Q_1^{(i)}$ for all $i = 1,2, \ldots ,s$. 
Suppose $P_2^{(i)} \subseteq Q_1^{(k)}$ for some $k \neq i$. Then $Q_1^{(k)}$ contains $P_{0,i}$ and $P_{0,k}$, a contradiction.  It follows that $P_2^{(i)} \not\subseteq Q_1^{(k)}$ for all $i = 1,2, \ldots ,s$ and for all $k = 1,2, \ldots ,s.$

By Lemma \ref{transcendental} using $X = \{Q_1^{(1)}, Q_1^{(2)}, \ldots ,Q_1^{(s)}\}$ there exists $\xt_2 \in \bigcap^s_{i = 1} P_2^{(i)}$ with $\xt_2 \not\in Q_1^{(i)}$ for all $i = 1,2, \ldots ,s$ such that $R_2 = R_1[\xt_2]_{(M \cap R_1[\xt_2])}$ is an N-subring of $T$ with $R_2 \cap P_2^{(i)} = \xt_2 R_2$ for every $i = 1,2, \ldots ,s.$ Moreover, prime elements of $R_1$ are prime in $R_2$.  In particular, $\xt_1$ is a prime element of $R_2$.  Note that $R_2$ is countable.

Repeat this process using Lemma \ref{findingheightones} with $X = \{Q_1^{(i)}, Q_2^{(i)}\}$ to find height one prime ideals $P_3^{(i)}$ and then Lemma \ref{transcendental} to find an element $\xt_3$ of $T$ and an N-subring $R_3$ of $T$. Continue inductively so that, if $P_{n - 1}^{(i)}$ for every $i = 1,2, \ldots ,s$, $\xt_{n - 1}$, and $R_{n - 1}$ have been defined, use Lemma \ref{findingheightones} with $X = \{Q_1^{(i)}, Q_2^{(i)}, \ldots ,Q_{n - 1}^{(i)}\}$ and Lemma \ref{transcendental} with $X = \bigcup_{i = 1}^s \{Q_1^{(i)}, Q_2^{(i)}, \ldots ,Q_{n - 1}^{(i)}\}$ to define the following:
\begin{itemize}
  \item height one prime ideals $P_n^{(i)}$ of $T$ for each $i = 1,2, \ldots ,s$ satisfying $P_n^{(i)} \subseteq Q_{n}^{(i)}$ and, if $P \in \mbox{Ass}(T)$ or $P \in \mbox{Ass}(T/rT)$ for some $0 \neq r \in R_{n - 1}$, then $P_n^{(i)} \not\subseteq P$. Moreover, for all $i = 1,2, \ldots ,s$ and for all $k = 1,2, \ldots ,s$, $P_n^{(i)} \not\subseteq Q_j^{(k)}$ whenever $j < n$.
  \item an element $\xt_n$ of $T$ such that $\xt_n \in \bigcap^s_{i = 1} P_n^{(i)}$ and, for all $i = 1,2, \ldots ,s$ we have $\xt_n \not\in Q_j^{(i)}$ whenever $j < n$.
  \item a countable N-subring $R_n$ of $T$ containing $\xt_1, \xt_2, \ldots ,\xt_n$ such that $R_n \cap P_n^{(i)} = \xt_n R_n$ for every $i = 1,2, \ldots ,s.$ Moreover, prime elements of $R_{n - 1}$ are prime in $R_n$.  In particular, $\xt_1, \xt_2, \ldots ,\xt_n$ are prime elements of $R_n$.
\end{itemize}
By Lemma \ref{union}, $S = \bigcup_{i = 1}^{\infty} R_n$ is a countable N-subring of $T$ and $\xt_n$ is a prime element of $S$ for every $n \in \mathbb{N}$.

At this point, we use the construction in \cite{heitmann} to build our UFD $A$. In particular, in the proof of Theorem 8 in \cite{heitmann}, one starts with a complete local ring $\tilde{T}$ and a localization of the prime subring of $\tilde{T}$ and then constructs a UFD whose completion is $\tilde{T}$. For our proof, we replace the localization of the prime subring of $T$ in the proof of Theorem 8 in \cite{heitmann} with the N-subring $S$ above. We then follow the proof of Theorem 8 in \cite{heitmann} to construct a UFD $(A,A \cap M)$. In particular, $A$ contains $S$, $\widehat{A} \cong T$, and prime elements of $S$ are prime in $A$. Since $A$ contains $S$, it contains a generating set for each $Q_n^{(i)}$ and so $(A \cap Q_n^{(i)})T = Q_n^{(i)}$. Also note that $\xt_n$ is a prime element of $A$ for every $n \in \mathbb{N}$ and, for every $i = 1,2, \ldots ,s$ and for every $n \in \mathbb{N}$, we have $A \cap P_n^{(i)} = \xt_n A$.
Because $T/Q_n^{(i)} = T/(A \cap Q_n^{(i)})T$ is the completion of $A/(A \cap Q_n^{(i)})$, we have $1 = \mbox{dim}(T/Q_n^{(i)}) = \mbox{dim}(A/(A \cap Q_n^{(i)})$ for every $i = 1,2, \ldots ,s$ and for every $n \in \mathbb{N}$.

For all $n \in \mathbb{N}$, define $J_{1,n} = P_n^{(i)} \cap A = \xt_nA$ and note that $J_{1,n} = J_{1,m}$ if and only if $n = m$. We now define $J_{2,n}^{(i)}$ for all $i = 1,2, \ldots ,s$ and for all $n \in \mathbb{N}$.
Observe that $\xt_nA \subseteq Q_n^{(i)} \cap A$. Suppose we have $\xt_nA = Q_n^{(i)} \cap A$. Then there is a $y_n \in M \cap A$ with $y_n \not\in Q_n^{(i)} \cap A = \xt_nA$. In this case, $M \cap A$ is a minimal prime ideal of $(\xt_n,y_n)A$ and so, by the generalized principal ideal theorem, ht$(M \cap A) \leq 2$. It follows that dim$T = \mbox{dim}A \leq 2$, a contradiction. Hence $\xt_nA \subsetneq Q_n^{(i)} \cap A$. If $Q_n^{(i)} \cap A = Q_n^{(k)} \cap A$ where $i \neq k$, then $Q_n^{(i)} = Q_n^{(k)}$, a contradiction. Thus, by prime avoidance, there exists $p^{(i)}_{n} \in A$ such that $p^{(i)}_{n} \in Q_n^{(i)}$, $p^{(i)}_{n} \not\in \xt_nA$ and $p^{(i)}_{n} \not\in Q_n^{(k)}$ for $k \neq i$. Let $Q^{(i)}_{2,n}$ be a minimal prime ideal of $P_n^{(i)} + p^{(i)}_{n}T$ that is contained in $Q_n^{(i)}$ and note that $P_n^{(i)} \subsetneq Q^{(i)}_{2,n}$ is saturated. In particular, ht$Q^{(i)}_{2,n} = 2$. Define $J_{2,n}^{(i)} = Q^{(i)}_{2,n} \cap A$ and note that, since $T$ is a faithfully flat extension of $A$, ht$J_{2,n}^{(i)} \leq 2$. It follows that $(0) \subsetneq J_{1,n} \subsetneq J_{2,n}^{(i)}$ is saturated. Also observe that $J_{2,n}^{(i)} = J_{2,n}^{(k)}$ if and only if $i = k$.

In the case that $J_{2,n}^{(i)} = Q_n^{(i)} \cap A$, we have completed defining our chain.  So suppose $J_{2,n}^{(i)} \subsetneq Q_n^{(i)} \cap A.$ We define $Q_{t,n}^{(i)}$ and $J_{t,n}^{(i)}$ inductively for $t \geq 3$. Let $q^{(i)}_{t,n} \in Q_n^{(i)} \cap A$ with $q^{(i)}_{t,n} \not\in J_{t - 1,n}^{(i)},$ and let $Q_{t,n}^{(i)}$ be a minimal prime ideal of $Q_{t - 1,n}^{(i)} + q^{(i)}_{t,n}T$ that is contained in $Q_n^{(i)}$. Define $J_{t,n}^{(i)} = Q_{t,n}^{(i)} \cap A$. Continue until $J_{t,n}^{(i)} = Q_n^{(i)} \cap A.$  Note that if $t < \ell$, then $J_{t,n}^{(i)} \subsetneq J_{\ell,n}^{(i)}$.

Observe that $P_{0,i} \subsetneq P^{(i)}_n \subsetneq Q^{(i)}_{2,n} \subsetneq \cdots \subsetneq Q^{(i)}_{\ell,n} \subsetneq M$, where $Q^{(i)}_{\ell,n}$ is the last element in our chain from above, is saturated. Since $T$ is catenary, $\ell = n_i - 1$. By the going down property, the chain $(0) \subsetneq J_{1, n} \subsetneq J^{(i)}_{2, n} \subsetneq \dots \subsetneq J^{(i)}_{n_i - 1, n} \subsetneq M \cap A$ is saturated.



Suppose $J^{(i)}_{a,b} = J^{(k)}_{c, d}$ for some $b \neq d$.  Without loss of generality, assume that $b < d$. Then $\xt_d \in J^{(i)}_{a,b} \subseteq Q^{(i)}_b$, a contradiction, and so we have $b = d$ and $J^{(i)}_{a,b} = J^{(k)}_{c, b}$. Then $p_b^{(i)} \in Q_b^{(k)}$. If $i \neq k$ then this contradicts the way $p_b^{(i)}$ was chosen. It follows that $i = k$, and thus, $a = c$ as well.
\end{proof}

Theorem \ref{constructionalt} is most interesting when applied to a complete local ring that is not equidimensional.  In this case, the local UFD $A$ given by the theorem has infinitely many height one prime ideals $P$ such that $A/P$ is not catenary. More specifically, for each of these prime ideals $P$, and for every $i = 1,2, \ldots ,s$, $A/P$ has saturated chains of prime ideals of length $n_i - 1$ that start at the zero ideal of $A/P$ and end at the maximal ideal of $A/P$. We end this section with an example.

\begin{exa}\label{motivatingexample}
Let $T = k[[x,y,z,w,t]]/((x) \cap (y,z))$ where $k$ is a field. Then $T$ satisfies the conditions of Theorem \ref{constructionalt}.  The minimal prime ideals of $T$ are $(x)$ and $(y,z)$, and we have dim$(T/(x)) = 4$ and dim$(T/(y,z)) = 3$. By Theorem \ref{constructionalt}, $T$ is the completion of a UFD $A$ such that $A$ has infinitely many height one prime ideals $\{J_n\}_{n \in \mathbb{N}}$ satisfying the condition that, for every $n \in \mathbb{N}$, there is a saturated chain of prime ideals of length 3 starting at $J_n$ and ending at the maximal ideal of $A$, and there is a saturated chain of prime ideals of length 2 starting at $J_n$ and ending at the maximal ideal of $A$.  Moreover, all of these chains are disjoint (except at the maximal ideal of $A$).  As a consequence, $A/J_n$ is not catenary for every $n \in \mathbb{N}$. We note that one has some choice for the elements of the chains having coheight one. In the proof of Theorem \ref{constructionalt}, one can choose the prime ideals $Q_n^{(i)}$ to satisfy the required conditions.  The elements of the chains in $A$ of coheight one will be the prime ideals $Q_n^{(i)} \cap A$ of $A$. For example, for our given $T$, let $\{\alpha_n\}_{n \in \mathbb{N}}$ be distinct elements of $\mathbb{C}$. One could choose $Q_n^{(1)}$ to be $(x,y,w,t + \alpha_n z)$ and $Q_n^{(2)}$ to be $(y,z,w,t + \alpha_n x)$. In this case, the coheight one ideals in the chains of length 3 will be $(x,y,w,t + \alpha_n z) \cap A$ and the coheight one ideals in the chains of length 2 will be $(y,z,w,t + \alpha_n x) \cap A$. Furthermore, the generator $\xt_n$ of the height one prime ideal $J_n$ of $A$ will be a regular element of $(x,y,w,t + \alpha_n z) \cap (y,z,w,t + \alpha_n x)$.
\end{exa}

\section{A Very Noncatenary UFD} \label{noncatenaryUFD}

In this section, we construct a dimension four countable local UFD $A$ such that, for \textit{every} height one prime ideal $P$ of $A$, the ring $A/P$ is not catenary. To do this, we start with the complete local ring $$T = \mathbb{Q}[[x,y,z,w,t]]/((x) \cap (y,z)).$$ The ring $A$ will be a subring of $T$ with $\widehat{A} \cong T$. Therefore, $T$ is a faithfully flat extension of $A$, and this fact will help us show that $A$ satisfies our desired property. To show that $\widehat{A} \cong T$, we use the following result.

\begin{prop}[\cite{heitmann1994}, Proposition 1] \label{completionmachine}
Let $(R,R \cap M)$ be a quasi-local subring of a complete local ring $(T,M)$ such that the map $R \longrightarrow T/M^2$ is onto and $IT \cap R = I$ for every finitely generated ideal $I$ of $R$.  Then $R$ is Noetherian and the natural homomorphism $\widehat{R} \longrightarrow T$ is an isomorphism.
\end{prop}

To show that $\widehat{A} \cong T$ using Proposition \ref{completionmachine}, we guarantee that the map $A \longrightarrow T/M^2$ is onto and $IT \cap A = I$ for every finitely generated ideal $I$ of $A$. We use the next lemma when constructing $A$ to ensure that is satisfies these two properties.  

\begin{lem}[\cite{simpson}, Lemma 3.7] \label{Austynpaper}
Let $(T,M)$ be a complete local ring with depth$T \geq 2$ and let $\mathfrak{p}$ be a nonmaximal prime ideal of $T$. Let $(R,R \cap M)$ be an infinite N-subring of $T$ with $R \cap \mathfrak{p} = (0)$ and let $u \in T$. Then there exists an N-subring $(S, S \cap M)$ of $T$ such that
\begin{enumerate}
    \item $R \subseteq S \subseteq T,$
    \item $u + M^2$ is in the image of the map $S \longrightarrow T/M^2$,
    \item $|S| = |R|,$
    \item $S \cap \mathfrak{p} = (0)$,
    \item prime elements of $R$ are prime in $S$, and
    \item for every finitely generated ideal $I$ of $S$, we have $IT \cap S = I$.
\end{enumerate}
\end{lem}

In the next lemma, we show that, if $T$ is a complete local ring satisfying certain conditions, and $R$ is a countable N-subring of $T$, then we can enlarge $R$ to another countable N-subring $A$ of $T$ whose completion is $T$. Moreover, for a given nonmaximal ideal $\mathfrak{p}$ of $T$, if $R \cap \mathfrak{p} = (0)$, then $A\cap \mathfrak{p} = (0)$ as well.

\begin{lem}\label{bigUFDtheorem}
Let $(T,M)$ be a complete local ring containing the rationals with $T/M$ countable and depth$T \geq 2$. Let $\mathfrak{p}$ be a nonmaximal prime ideal of $T$. Let $(R,R \cap M)$ be a countable N-subring of $T$ with $R \cap \mathfrak{p} = (0)$. Then there exists a countable N-subring $(A, A \cap M)$ of $T$ such that $R \subseteq A$, prime elements of $R$ are prime in $A$, $A$ is Noetherian, $\widehat{A} \cong T$, and $A \cap \mathfrak{p} = (0)$.
\end{lem}

\begin{proof}
Since $T/M$ is countable, $T/M^2$ is countable. Enumerate the elements of $T/M^2$ as $u_0 + M^2, u_1 + M^2, u_2 + M^2, \ldots ,$. Let $R_0 = \mathbb{Q}$ and note that $R_0$ is an infinite N-subring of $T$ with $R_0 \cap \mathfrak{p} = (0)$. Let $R_1$ be the N-subring of $T$ obtained from Lemma \ref{Austynpaper} so that $R_0 \subseteq R_1$, $u_0 + M^2$ is in the image of the map $R_1 \longrightarrow T/M^2$, $|R_1| = |R_0|$, $R_1 \cap \mathfrak{p} = (0)$, prime elements of $R_0$ are prime in $R_1$, and, for every finitely generated ideal $I$ of $R_1$, we have $IT \cap R_1 = I$. Use Lemma \ref{Austynpaper} using $R_1$ and $u_1$ to define $R_2$ so that $R_1 \subseteq R_2$, $u_1 + M^2$ is in the image of the map $R_1 \longrightarrow T/M^2$, $|R_2| = |R_1|$, $R_2 \cap \mathfrak{p} = (0)$, prime elements of $R_1$ are prime in $R_2$, and, for every finitely generated ideal $I$ of $R_2$, we have $IT \cap R_2 = I$. Continue to define $R_n$ for every $n \geq 3$ so that $R_{n - 1} \subseteq R_n$, $u_{n - 1} + M^2$ is in the image of the map $R_n \longrightarrow T/M^2$, $|R_n| = |R_{n - 1}|$, $R_n \cap \mathfrak{p} = (0)$, prime elements of $R_{n - 1}$ are prime in $R_n$, and, for every finitely generated ideal $I$ of $R_n$, we have $IT \cap R_n = I$.

Let $A = \bigcup_{j = 0}^{\infty} R_j$ and note that $R \subseteq A$. By Lemma \ref{union}, $A$ is a countable N-subring of $T$ and prime elements of $R$ are prime in $A$. By construction, $A \cap \mathfrak{p} = (0)$ and the map $A \longrightarrow T/M^2$ is onto. Now let $I$ be a finitely generated ideal of $A$ and let $c \in IT \cap A$. We have $I = (a_1, a_2, \ldots ,a_m)$ for some $a_i \in A$. Choose $N$ so that $c, a_1, a_2, \ldots ,a_m \in R_N$. Then $c \in (a_1, \ldots ,a_m)T \cap R_N = (a_1, \ldots ,a_m)R_N \subseteq I.$ It follows that $IT \cap A = I$. By Proposition \ref{completionmachine}, we have that $A$ is Noetherian and $\widehat{A} \cong T.$
\end{proof}

As mentioned previously, we begin our construction with the complete local ring $$T = \mathbb{Q}[[x,y,z,w,t]]/((x) \cap (y,z)).$$ The next three results demonstrate facts about this ring that we use to construct our final UFD.

\begin{lem}\label{findgoodQ}
Let $a$ be a nonzero regular element of the complete local ring $$T = \mathbb{Q}[[x,y,z,w,t]]/((x) \cap (y,z)).$$
Suppose that the ideal $(a,x)$ has a minimal prime ideal $P_1$ that does not contain $(y,z)$ and that the ideal $(a,y,z)$ has a minimal prime ideal $P_2$ that does not contain $(x)$. Then there exist prime ideals $Q_1$ and $Q_2$ of $T$ such that, for $i = 1,2$, we have $P_i \subsetneq Q_i$, $Q_i$ only contains one minimal prime ideal of $T$ and dim$(T/Q_i) = 1$.
\end{lem}

\begin{proof}
Let $M = (x,y,z,w,t)$. Note that $P_1$ and $P_2$ contain only one minimal prime ideal of $T$ and that the chains $(x) \subsetneq P_1$ and $(y,z) \subsetneq P_2$ are saturated. Since $T$ is Noetherian and catenary, there are infinitely many prime ideals strictly between $P_2$ and $M$.  If $Q$ is such a prime ideal containing $(x)$, then $Q$ is a minimal prime ideal of $P_2 + (x)$, of which there are only finitely many.  Therefore, we can choose $Q_2$ to satisfy the conditions that $P_2 \subsetneq Q_2 \subsetneq M$ and $Q_2$ does not contain $(x)$.

There exists a prime ideal $J$ of $T$ such that $P_1 \subsetneq J \subsetneq M$ and dim$(T/J) = 1$. Since $T$ is catenary, $P_1 \subsetneq J$ is not saturated. There are infinitely many prime ideals strictly between $P_1$ and $J$.  If $I$ is such a prime ideal containing $(y,z)$, then $I$ is a minimal prime ideal of $P_1 + (y,z)$ of which there are only finitely many. Thus, there is a prime ideal $I$ of $T$ satisfying the conditions that $P_1 \subsetneq I \subsetneq J \subsetneq M$ and $I$ does not contain $(y,z)$. By a similar argument replacing $P_1$ by $I$, there exists a prime ideal $Q_1$ of $T$ satisfying the conditions that $P_1 \subsetneq I\subsetneq Q_1 \subsetneq M$ and $Q_1$ does not contain $(y,z)$.
\end{proof}

\begin{lem}\label{minimals}
Let $a$ be an element of the complete local ring $T = \mathbb{Q}[[x,y,z,w,t]]/((x) \cap (y,z))$
satisfying the condition that $a \not\in (x,y,z)$. Then every minimal prime ideal of the ideal $(a,x)$  does not contain $(y,z)$ and every minimal prime ideal of the ideal $(a,y,z)$ does not contain $(x)$.
\end{lem}

\begin{proof}
Let $P_1$ be a minimal prime ideal of $(a,x)$. Then in the ring $T/(x)$, the principal ideal theorem gives us that ht$(P_1/(x)) = 1$. Suppose $(y,z) \subseteq P_1$. Then $(x,y,z) \subseteq P_1$ and so in the ring $T/(x)$, $(x,y,z)/(x) \subseteq P_1/(x)$, and it follows that ht$((x,y,z)/(x) )\leq 1$. But $(x,y)/(x)$ is a prime ideal of $T/(x)$ and we have $(x)/(x) \subsetneq (x,y)/(x) \subsetneq (x,y,z)/(x)$, a contradiction.

Now let $P_2$ be a minimal prime ideal of $(a,y,z)$ and suppose that $(x) \subseteq P_2$. Then $(x,y,z) \subseteq P_2$. In the ring $T/(y,z)$, both of the prime ideals $P_2/(y,z)$ and $(x,y,z)/(y,z)$ have height one, and so $P_2/(y,z) = (x,y,z)/(y,z)$. Thus, $a + (y,z) \in (x,y,z)/(y,z)$ and we have that $a \in (x,y,z)$, a contradiction.
\end{proof}

\begin{lem}\label{ExampleQ}
Let $Q$ be a prime ideal of the complete local ring $T = \mathbb{Q}[[x,y,z,w,t]]/((x) \cap (y,z))$ satisfying dim$(T/Q) = 1$. Then, if $v$ is a nonzero regular element of $T$ and $P \in \mbox{Ass}(T/vT)$, we have $Q \not\subseteq P$.
\end{lem}

\begin{proof}
Note that depth$T = 3$ and so if $v$ is a nonzero regular element of $T$ then  $(x,y,z,w,t) \not\in \mbox{Ass}(T/vT)$. So suppose that $Q \in \mbox{Ass}(T/vT)$ for some nonzero regular element $v$ of $T$. Then Theorem 17.2 in \cite{matsumura} gives us that the dimension of the ring $(T/vT)/(Q/vT) \cong T/Q$ is greater than or equal to the depth of the ring $T/vT$. This implies that the depth of $T/vT$ is at most one, contradicting that the depth of $T$ is 3.
\end{proof}

We are now equipped with the tools needed to prove the main result of this section.

\begin{thm}\label{everyheightone}
Let $T = \mathbb{Q}[[x,y,z,w,t]]/((x) \cap (y,z))$. There exists a local UFD $A$ such that $\widehat{A} \cong T$ and such that, for every height one prime ideal $P$ of $A$, $A/P$ is not catenary.
\end{thm}

\begin{proof}
Let $R_0 = \mathbb{Q}$, let $\mathfrak{p} = (x,y,z)$, and let $M = (x,y,z,w,t)$. Observe that $R_0$ is an infinite N-subring of $T$ and $R_0 \cap \mathfrak{p} = (0)$. Use Lemma \ref{bigUFDtheorem} to find a countable N-subring $A_0$ of $T$ such that $A_0$ is Noetherian, $\widehat{A}_0 \cong T$, and $A_0 \cap \mathfrak{p} = (0)$. Since $\widehat{A}_0 \cong T$, the map $A_0 \longrightarrow T/M^2$ is onto and $IT \cap A_0 = I$ for every finitely generated ideal $I$ of $A_0$.

Because $A_0$ is a Noetherian UFD, all of its height one prime ideals are principal.  Enumerate the height one prime ideals of $A_0$ by $a_1A_0, a_2A_0, \ldots ,a_nA_0, \ldots.$ Note that, for all $j \geq 1$, we have $a_j \not\in \mathfrak{p}$. For $j \geq 1$, let $P_{1,a_j}$ be a minimal prime ideal (in $T$) of $(a_j,x)$ and let $P_{2,a_j}$ be a minimal prime ideal of $(a_j,y,z)$. By Lemma \ref{minimals}, $P_{1,a_j}$ does not contain $(y,z)$ and $P_{2,a_j}$ does not contain $(x)$. By Lemma \ref{findgoodQ}, there are prime ideals $Q_{1,a_j}$ and $Q_{2,a_j}$ of $T$ such that $P_{1,a_j} \subseteq Q_{1,a_j}$, $P_{2,a_j} \subseteq Q_{2,a_j}$, $Q_{1,a_j}$ and $Q_{2,a_j}$ only contain one minimal prime ideal of $T$, dim$(T/Q_{1,a_j}) = 1$, and dim$(T/Q_{2,a_j}) = 1$. Since dim$(T/\mathfrak{p})> 1$, we have $Q_{1,a_j} \not\subseteq \mathfrak{p}$ and $Q_{2,a_j} \not\subseteq \mathfrak{p}$ for all $j \geq 1.$ Note that if $P \in  \mbox{Ass}(T) = \{(x), (y,z)\}$, then $Q_{1,a_j} \not\subseteq P$ and $Q_{2,a_j} \not\subseteq P$ for all $j \geq 1.$ By Lemma \ref{ExampleQ}, if $v$ is a nonzero regular element of $T$ and $P \in \mbox{Ass}(T/vT)$, we have $Q_{1,a_j} \not\subseteq P$ and $Q_{2,a_j} \not\subseteq P$ for every $j \geq 1.$  Use Lemma \ref{generatingsets} to find a countable N-subring $R_1$ of $T$ such that $A_0 \subseteq R_1$, $R_1 \cap \mathfrak{p} = (0)$, prime elements in $A_0$ are prime in $R_1$, and for every $j \in \mathbb{N}$, $R_1$ contains a generating set for $Q_{1,a_j}$ and for $Q_{2,a_j}$. 

By Lemma \ref{bigUFDtheorem} there exists a countable N-subring $A_1$ of $T$ such that $R_1 \subseteq A_1$, prime elements in $R_1$ are prime in $A_1$, $A_1$ is Noetherian, $\widehat{A}_1 \cong T$, and $A_1 \cap \mathfrak{p} = (0)$. Note that $IT \cap A_1 = I$ for every finitely generated ideal $I$ of $A_1$. Also note that prime elements in $A_0$ are prime in $A_1$. Repeat the procedure in the previous paragraph replacing $A_0$ with $A_1$ to find a countable N-subring $R_2$ of $T$ such that $A_1 \subseteq R_2$, $R_2 \cap \mathfrak{p} = (0)$ and prime elements in $A_1$ are prime in $R_2$. Moreover, for every height one prime ideal $aA_1$ of $A_1$, $R_2$ contains a generating set for prime ideals $Q_{1,a}$ and $Q_{2,a}$ of $T$ where $Q_{1,a}$ contains a minimal prime ideal of $(a,x)$, $Q_{2,a}$ contains a minimal prime ideal of $(a,y,z)$, $Q_{1,a}$ and $Q_{2,a}$ contain only one minimal prime ideal of $T$, dim$(T/Q_{1,a}) = 1$, and dim$(T/Q_{2,a}) = 1$.  Continue to form a countably infinite chain of N-subrings $R_0 \subseteq A_0 \subseteq R_1 \subseteq A_1, \subseteq \cdots$ of $T$.

Define $A = \bigcup_{j = 0}^{\infty}A_j.$ By Lemma \ref{union}, $A$ is a countable N-subring of $T$ and elements that are prime in $A_j$ for some $j \geq 0$ are also prime in $A$. Since $A$ is an N-subring, it is a UFD. By construction, the map $A \longrightarrow T/M^2$ is onto and $IT \cap A = I$ for every finitely generated ideal $I$ of $A$. By Proposition \ref{completionmachine}, $A$ is Noetherian and $\widehat{A} \cong T$. 

Let $P$ be a height one prime ideal of $A$. Since $A$ is a UFD, $P$ is principal, and so $P = aA$ for some $a \in A$. Choose $N$ so that $a \in A_N$. Suppose that $a = p_1p_2 \cdots p_m$ is the prime factorization of $a$ in $A_N$. Then $p_1, p_2, \ldots ,p_m$ are all prime elements in $A$, and so $a = p_1p_2 \ldots p_m$ is also the prime factorization of $a$ in $A$.  As $a$ is prime in $A$, we have that $m = 1$ and $a = p_1$. It follows that $a$ is prime in $A_N$. Thus $A_{N+ 1}$, and hence $A$, contains a generating set for prime ideals $Q_{1}$ and $Q_{2}$ of $T$ where $Q_{1}$ contains a minimal prime ideal $P_1$ of $(a,x)$, $Q_{2}$ contains a minimal prime ideal $P_2$ of $(a,y,z)$, $Q_{1}$ and $Q_{2}$ contain only one minimal prime ideal of $T$, dim$(T/Q_{1}) = 1$, and dim$(T/Q_{2}) = 1$. Note that $P_1$ and $P_2$ are both height one prime ideals of $T$.
Thus, ht$(A \cap P_1) = 1$ and ht$(A \cap P_2) = 1$. Therefore, $A \cap P_1 = aA$ and $A \cap P_2 = aA$. The rest of the argument is similar to the argument in the proof of Theorem \ref{constructionalt}, and so we omit some of the detailed explanations. The completion of $A/(A \cap Q_1)$ is $T/Q_1$ and it follows that dim$(A/(A \cap Q_1)) = 1$. Similarly, dim$(A/(A \cap Q_2)) = 1$. Therefore, $aA \subsetneq A \cap Q_1$ and $aA \subsetneq A \cap Q_2.$ Let $b \in A \cap Q_1$ with $b \not\in aA$ and let $c \in A \cap Q_2$ with $c \not\in aA$. Let $J_1$ be a minimal prime ideal of $P_1 + bT$ contained in $Q_1$ and let $J_2$ be a minimal prime ideal of $P_2 + cT$ contained in $Q_2$. Then $J_2 = Q_2$, $(0) \subsetneq aA \subsetneq A \cap J_1$ is saturated and $(0) \subsetneq aA \subsetneq A \cap J_2 = A \cap Q_2 \subsetneq A \cap M$ is saturated. We now have $J_1 \subsetneq Q_1$ and so $A \cap J_1 \subsetneq A \cap Q_1$. Let $d \in A \cap Q_1$ with $d \not\in J_1$, and let $I$ be a minimal prime ideal of $J_1 + dT$ contained in $Q_1$. Then $I = Q_1$ and we have that $(0) \subsetneq aA \subsetneq A \cap J_1 \subsetneq A \cap I = A \cap Q_1 \subsetneq A \cap M$ is saturated. It follows that $A/aA = A/P$ is not catenary.
\end{proof}



\newpage
\bibliography{refs}

\begin{thebibliography}{10}

\bibitem{small}
Chloe~I. Avery, Caitlyn Booms, Timothy~M. Kostolansky, S.~Loepp, and Alex
  Semendinger.
\newblock Characterization of completions of noncatenary local domains and
  noncatenary local {UFD}s.
\newblock {\em J. Algebra}, 524:1--18, 2019.

\bibitem{colbert2022finite}
C.~Colbert and S.~Loepp.
\newblock Every finite poset is isomorphic to a saturated subset of the
  spectrum of a {N}oetherian {UFD}.
\newblock {\em J. Algebra}, 643:340--370, 2024.

\bibitem{heitmann1979}
Raymond~C. Heitmann.
\newblock Examples of noncatenary rings.
\newblock {\em Trans. Amer. Math. Soc.}, 247:125--136, 1979.

\bibitem{heitmann}
Raymond~C. Heitmann.
\newblock Characterization of completions of unique factorization domains.
\newblock {\em Trans. Amer. Math. Soc.}, 337(1):379--387, 1993.

\bibitem{heitmann1994}
Raymond~C. Heitmann.
\newblock Completions of local rings with an isolated singularity.
\newblock {\em J. Algebra}, 163(2):538--567, 1994.

\bibitem{loepp}
S.~Loepp.
\newblock Constructing local generic formal fibers.
\newblock {\em J. Algebra}, 187(1):16--38, 1997.

\bibitem{simpson}
S.~Loepp and Austyn Simpson.
\newblock Noncatenary splinters in prime characteristic, arXiv:2401.00925.

\bibitem{semendinger}
Susan Loepp and Alex Semendinger.
\newblock Maximal chains of prime ideals of different lengths in unique
  factorization domains.
\newblock {\em Rocky Mountain J. Math.}, 49(3):849--865, 2019.

\bibitem{matsumura}
Hideyuki Matsumura.
\newblock {\em Commutative ring theory}, volume~8 of {\em Cambridge Studies in
  Advanced Mathematics}.
\newblock Cambridge University Press, Cambridge, 1986.
\newblock Translated from the Japanese by M. Reid.

\bibitem{nagata_1956}
Masayoshi Nagata.
\newblock On the chain problem of prime ideals.
\newblock {\em Nagoya Math. J.}, 10:51--64, 1956.

\bibitem{wiegands}
Roger Wiegand and Sylvia Wiegand.
\newblock Prime ideals in {N}oetherian rings: a survey.
\newblock In {\em Ring and module theory}, Trends Math., pages 175--193.
  Birkh\"{a}user/Springer Basel AG, Basel, 2010.

\end{thebibliography}
\bibliographystyle{plain}
\end{document}